\documentclass[reqno]{amsart}
\usepackage{amssymb,amsmath}
\usepackage{amsthm}
\usepackage{color,graphicx}

\usepackage[pagewise,mathlines,displaymath]{lineno}

\newcommand{\R}{\mathbb R}

\newcommand{\Z}{\mathbb Z}

\newcommand{\sgn}{\text{sgn}}

\newcommand{\trinorm}{|\!|\!|}

\newtheorem{theorem}{Theorem}[section]

\newtheorem{remark}[theorem]{Remark}
\newtheorem{lemma}[theorem]{Lemma}
\newtheorem{corollary}[theorem]{Corollary}

\begin{document}

\vglue-1cm \hskip1cm
\title[Dispersion Generalized Benjamin-Ono Equation]{Global well-posedness for the $k$-dispersion generalized Benjamin-Ono equation}

\author[L. G. Farah]{Luiz G. Farah}
\address{ICEx, Universidade Federal de Minas Gerais, Av. Ant\^onio Carlos, 6627, Caixa Postal 702, 30123-970,
Belo Horizonte-MG, Brazil}
\email{lgfarah@gmail.com}

\author[F. Linares]{Felipe Linares}
\address{IMPA, Estrada Dona Castorina 110, 22460-320, Rio de Janeiro-RJ,
 Brazil.}
\email{linares@impa.br}

\author[A. Pastor]{Ademir Pastor}
\address{IMECC-UNICAMP, Rua S\'ergio Buarque de Holanda, 651, 13083-859, Cam\-pi\-nas-SP, Bra\-zil}
\email{apastor@ime.unicamp.br}


\date{23/10/2012}


\begin{abstract}
 We consider the $k$-dispersion generalized Benjamin-Ono equation in the supercritical case. We establish sharp
 conditions on the data to show global well-posedness in the energy space for this family of nonlinear dispersive equations. We also
 prove similar results for the generalized Benjamin-Ono equation.
 \end{abstract}

\maketitle


\maketitle

\section{Introduction}

In this paper we continue the study on sharp conditions to obtain
global solutions for supercritical nonlinear dispersive models. With this aim
we consider the  initial value problem (IVP)
for the  $k$-dispersion generalized Benjamin-Ono (DGBO) equation
\begin{equation}\label{DGBO}
\begin{cases}
\partial_t u -   D^{\beta}\partial_x u +\partial_x (u^{k+1}) = 0, \quad x\in \R, \;t>0,\;\; 1\leq \beta <2,\;\;k\in\Z^{+},\\
u(x,0) = u_0(x),
\end{cases}
\end{equation}
where  $D^{\beta}$ denotes  the homogeneous derivative of order $\beta \in\R$,
$$
D^{\beta} f=c_{\beta}\big(|\xi|^{\beta}\widehat{f}\,\big)^{\vee}=(\mathcal H\,\partial_x)^{\beta}f,
$$
and  $\mathcal H$ denotes the Hilbert transform,
\begin{equation*}
\mathcal H f(x)=(-i\,\sgn(\xi) \widehat{f}(\xi))^{\vee}(x).
\end{equation*}
These equations arise as mathematical models for the unidirectional propagation of weakly nonlinear dispersive long waves.
Notice that when $\beta=2$ and $k=1$ the equation in \eqref{DGBO} is the well known Korteweg-de Vries (KdV) equation and
when $\beta=1$ and $k=1$ it is the Benjamin-Ono (BO) equation. (See \cite{KdV}, \cite{Be} and \cite{On} for their derivation and
\cite{KePoVe2}, \cite{CKSTT}, \cite{IoKe} for the sharpest local and global well-posedness results). 

We shall mention that recently several developments regarding the local well-posedness has been obtained for both BO and DGBO equations. 
One of the interesting features of the IVP associated to these families of equations is the obstruction to its solvability by iteration methods. Indeed, Molinet, Saut and Tzvetkov in \cite{MoSaTz} proved that the contraction mapping principle cannot be applied to solve these equations in the standard Sobolev spaces. For more details concerning well-posedness  we refer to  \cite {KePoVe}, \cite{CoKeSt}, \cite{MoRi2},  \cite{He},
\cite{Gu2}, \cite{HeIoKeKo}, \cite{FLP1}, \cite{FLP2},  \cite{MoRi} for
 the GDBO equation  and  \cite{BuPl}, \cite{FP}, \cite{FLP3}, \cite{Io}, \cite{IoKe}, \cite{KeKo}, \cite{KoTz},  \cite{MoPi},
\cite{Po}, \cite{Sa}, \cite{Ta} for the BO equation.

Regarding local well-posedness for the IVP \eqref{DGBO}, $k\ge2$, in \cite{KePoVe}  Kenig, Ponce and Vega 
showed that the IVP \eqref{DGBO} is locally well-posed in $H^s(\R)$,  for $s\ge(9-3\beta)/4$ and globally well-posed
in $H^{\beta/2}(\R)$ for $\beta\ge9/5$.

Recently, Kenig, Martel and Robiano \cite{KMR} considered the family of dispersion generalized Benjamin-Ono equations,
\begin{equation}\label{cri-dgbo}
\partial_t u-D^{\beta}\partial_x u+|u|^{2\beta}\partial_xu=0,
\end{equation}
which are critical with respect to both the $L^2$ norm and the global existence of solutions. Observe that these equations are the
interpolated ones between the modified BO equation ($\beta=1$) and the critical KdV equation $\beta=2$.  In \cite{KMR}
it was proved among other results, the local well-posedness in the energy space $H^{\beta/2}(\R)$ for $\beta\in (1,2)$.
They also studied the blow-up problem in the same vein as in \cite{MM} and \cite{Mer}. In particular they established the finite time
blow-up of solutions with negative energy in the energy space when $\beta$ is close to 2.

Here, we are interested in global well-posedness for the DGBO equation in the supercritical case, that is, $k>2\beta$. Important relations in this study are the conserved quantities satisfied by  real solutions of the IVP \eqref{DGBO} 
\begin{equation}\label{MCd}
M(u(t))=\int_\mathbb{R} u^2(x,t)\, dx
\end{equation}
and
\begin{equation}\label{ECd}
E(u(t))=\dfrac{1}{2}\int_{\mathbb{R}}|D^{\beta/2} u(x,t)|^2\;dx
-\dfrac{1}{k+2}\int_{\mathbb{R}} u^{k+2}(x,t)\;dx.
\end{equation}
These quantities allow us to establish an {\it a priori} estimate in the energy space $H^{\beta/2}(\R)$.

Next we will describe our main results concerning  the case $\beta\in[1,2)$. 
First, we recall that in \cite{farah-lp} the authors studied  the IVP associated with the $L^2$-supercritical
generalized KdV equation ($\beta=2$ in equation \eqref{DGBO}). Under sharp conditions satisfied by the initial data they established global well-posedness in the energy space $H^1(\R)$.  
This result is the first step in view to the more general theory studied by Kenig and Merle \cite{KM}  and by 
Holmer and Roudenko \cite{HR}.  Our study below is based on these works.

One of the main tools needed for this analysis is having the best constant for the Gagliardo-Nirenberg inequality
\begin{equation*}
\|f\|_{L^{k+2}}^{k+2}\le K_{\rm opt}^{k+2}\,\|D^{\beta/2}
f\|_{L^2}^{k/\beta}\|f\|_{L^2}^{(2+(k+2)(\beta-1))/\beta},
\end{equation*}
where the constant $K_{\rm opt}$  depends on  $Q$ the unique positive, even, decreasing  (for $x>0$) solution of the equation
\begin{equation}\label{groundA}
D^\beta Q+Q-Q^{k+1}=0.
\end{equation}
The sharp constant was obtained in \cite{ABLS}. In \cite{We} the existence of the ground state $Q$ was showed. The uniqueness of $Q$ 
which is also key in our arguments was recently established by Frank and Lenzmann in \cite{FL}.

Our main result reads as follows.

\begin{theorem}\label{global_dgbo}
 Let $\beta \in [1,2)$ and $u_0\in H^{\beta/2}(\R)$. Assume  $k>2\beta$ and let  $s_k=1/2-\beta/k$. Suppose that
\begin{equation}\label{GR1d}
E(u_0)^{s_k} M(u_0)^{\beta/2-s_k} < E(Q)^{s_k} M(Q)^{\beta/2-s_k} , \,\,\,
E(u_0) \geq 0.
\end{equation}
If
\begin{equation}\label{GR2d}
\|D^{\beta/2} u_0\|_{L^2}^{s_k}\|u_0\|_{L^2}^{\beta/2-s_k} < \|D^{\beta/2}
Q\|_{L^2}^{s_k}\|Q\|_{L^2}^{\beta/2-s_k},
\end{equation}
then for any $t$ as long as the solution exists,
\begin{equation}\label{GR3d}
\|D^{\beta/2} u(t)\|_{L^2}^{s_k}\|u_0\|_{L^2}^{\beta/2-s_k}=\|D^{\beta/2}
u(t)\|_{L^2}^{s_k}\|u(t)\|_{L^2}^{\beta/2-s_k} <\|D^{1/2}
Q\|_{L^2}^{s_k}\|Q\|_{L^2}^{\beta/2-s_k},
\end{equation}
where $Q$ is the unique positive, even, decreasing (for $x>0$)   solution of
\eqref{groundA}.
\end{theorem}

\begin{corollary} \label{cor2}
Let $\beta \in (1,2)$. Under the hypotheses of Theorem \ref{global_dgbo} the local solutions  given in Theorem \ref{localdbo} below can be 
extended to any interval of time $[0,T]$.
\end{corollary}

Next we specialize ourselves in the case $\beta=1$, thus the equation in \eqref{DGBO} becomes the generalized BO equation for $k\ge 2$
and the energy space is $H^{1/2}(\R)$.  A scaling argument suggests that the best local  well-posedness for the IVP \eqref{DGBO} should be attained for $s>s_k=1/2-1/k$. We review the state of the art concerning this IVP. In the case $k=2$, Kenig and Takaoka \cite{KeTa} established local well-posedness for data in $H^{1/2}(\R)$ and, using the conserved quantities, global well-posedness in $H^{s}(\R)$, $s\ge 1/2$. For $k=3$ and $k\ge 4$, local-wellposedness 
was obtained  by Vento in \cite{V1} in $H^s(\R)$, $s>1/3$ and $s\ge s_k$, respectively. We shall observe that in these results the existence time depends on the initial data itself and not only on its norm, that is, $T=T(u_0)$. Therefore, we cannot use the conserved quantities to extend the local solutions to global ones by means of {\it a priori} estimates. However, if  $k\ge 5$, the local theory was proved in $H^{s}(\R)$, $s\ge 1/2$  with $T=T(\|u_0\|_{s})>0$ in \cite{MoRi}. Using these local solutions 
we can extend them globally as stated in the next result.
%
%

\begin{corollary}
Assume $k\geq5$ and the hypotheses of Theorem \ref{global_dgbo} for $\beta=1$. Then the local $H^{1/2}$-solution of the generalized Benjamin-Ono
(gBO) equation
\begin{equation}\label{genbo}
\left\{
\begin{array}{lll}
{\displaystyle u_t-\mathcal{H}\partial_x^2 u+ \partial_x(u^{k+1})  =  0,  }  \qquad x \in \mathbb{R}, \,\,\,\, t>0, \\
{\displaystyle  u(x,0)=u_0(x)},
\end{array}
\right.
\end{equation}
can be extended to any interval of time $[0,T]$.
\end{corollary}

\begin{remark} We observe that in the cases $k= 3,\,4$  it is an open problem to establish a local theory in the
energy space with time $T=T(\|u_0\|_{1/2})$.  With this result available Theorem \ref{global_dgbo}  will imply in
the existence of global solutions with the sharp conditions on the data.
\end{remark}

The paper is organized as follows: In Section 2 the local theory for the IVP \eqref{DGBO}, $\beta\in(1,2)$ and $k>2\beta$ will be established. Then, in Section 3, the proof of Theorem \ref{global_dgbo} will be given.

\section{Local well-posedness in the energy space}

Here we consider the Cauchy problem associated with the supercritical
dispersion generalized Benjamin-Ono (DGBO) equation
\begin{equation}\label{dBO}
\left\{
\begin{array}{lll}
{\displaystyle u_t-D^\beta\partial_x u+ \partial_x(u^{k+1})  =  0,  }  \qquad x \in \mathbb{R}, \,\,\,\, t>0, \\
{\displaystyle  u(x,0)=u_0(x)},
\end{array}
\right.
\end{equation}
where $u$ is a real-valued function, $1<\beta<2$ and $k>2\beta$ is an integer
number.

First, consider the linear IVP
\begin{equation}\label{a1}
     \left\{
\begin{array}{lll}
{\displaystyle u_t-D^\beta\partial_x u=  0,}  \qquad x \in \mathbb{R}, \,\,\,\, t \in \mathbb{R}, \\
{\displaystyle  u(x,0)=u_0(x)}.
\end{array}
\right.
\end{equation}
The solution of \eqref{a1} is given by the unitary group
$\{U_\beta(t)\}_{t=-\infty}^\infty$ such that
\begin{equation}\label{a2}
u(t)=U_\beta(t)u_0(x)= \int_{\mathbb{R}}
e^{i(t|\xi|^\beta\xi+x\xi)}\widehat{u}_0(\xi) d\xi.\\
\end{equation}

We start by recalling the following estimates proved in \cite{KMR}.

\begin{lemma}\label{lemmalinear}
Assume $1<\beta<2$. For $0<T<1$, there exists a constant $C>0$ such that
\begin{itemize}
    \item[(i)] For all $u_0\in L^2$, $$\|U_\beta(t)u_0\|_{L^\infty_TL^2_x}\leq C \|u_0\|_{L^2}.$$

    \item[(ii)] There exists $\gamma>0$ such that
    for all $u_0\in H^{\beta/2}$, $$\|\partial_xU_\beta(t)u_0\|_{L^\infty_xL^2_T}\leq C T^\gamma
    \|u_0\|_{H^{\beta/2}}.$$

    \item[(iii)] For all $u_0\in H^{s}$, where
    $s>\frac{3}{4}-\frac{\beta}{4}$,
    $$
    \|U_\beta(t)u_0\|_{L^{2\beta}_xL^\infty_T}\leq C \|u_0\|_{H^{s+}}.
    $$
    \item[(iv)] For all $g\in L^1_xL^2_T$,
    $$
    \|D^{\beta/2}\int_0^t
    U_\beta(t-t')g(\cdot,t')dt'\|_{L^\infty_TL^2_x} \leq
    C\|g\|_{L^1_xL^2_T}.
    $$
    \item[(v)] For $0\leq s <\beta$, there exists $\gamma>0$ such that
    $$
    \|D^{s}\int_0^t
    U_\beta(t-t')g(\cdot,t')dt'\|_{L^\infty_xL^2_T} \leq
    CT^\gamma\|g\|_{L^1_xL^2_T}.
    $$
    \item[(vi)] There exists $\gamma>0$ such that
    $$
    \|\int_0^t
    U_\beta(t-t')g(\cdot,t')dt'\|_{L^\infty_TL^2_x} \leq
    CT^\gamma\|g\|_{L^1_xL^2_T}.
    $$
\end{itemize}
\end{lemma}
\begin{proof}
See Lemma 1 in \cite{KMR}.
\end{proof}

\begin{corollary}\label{maximalcor}
Assume  $1<\beta<2$, $k>2\beta$ and let $1/2+$ denotes any number bigger than $1/2$. The following statements hold.
\begin{itemize}
    \item[(i)] For all $u_0\in H^{1/2+}$,
    $$
    \|U_\beta(t)u_0\|_{L^{k}_xL^\infty_T}\leq C \|u_0\|_{H^{1/2+}}
    $$
    \item[(ii)] For all $g\in L^1_xL^2_T$ such that $\partial_xg\in
    L^1_xL^2_T$,
    $$
    \|\int_0^tU_\beta(t-t')\partial_xg(\cdot,t')dt'\|_{L^{k}_xL^\infty_T}\leq
    C\|\partial_xg\|_{L^1_xL^2_T}.
    $$
\end{itemize}
\end{corollary}
\begin{proof}
From Lemma \ref{lemmalinear}-(iii), we have
\begin{equation}\label{lemiii}
\|U_\beta(t)u_0\|_{L^{2\beta}_xL^\infty_T}\leq C \|u_0\|_{H^{1/2+}}.
\end{equation}
Also, from Sobolev's embedding,
\begin{equation}\label{sobemb}
\|U_\beta(t)u_0\|_{L^{\infty}_xL^\infty_T}\leq C \|u_0\|_{H^{1/2+}}.
\end{equation}
Hence, (i) follows just interpolating \eqref{lemiii} and \eqref{sobemb}.

The proof of (ii) follows exactly as in \cite[Lemma 1-(xi)]{KMR}. Indeed,
let $P_n$ be the projection on frequencies $\simeq2^n$ and define
$$
T_ng(\cdot,t)=\int_0^tU_\beta(t-t')P_n\partial_xg(\cdot,t')dt',
$$
$$
\widetilde{T}_ng(\cdot,t)=\int_0^TU_\beta(t-t')P_n\partial_xg(\cdot,t')dt'.
$$
Choose a number $1/2+$ such that $1/2+<\beta/2$. Then, by (i), localization
in frequencies, and Lemma \ref{lemmalinear}-(iv), we deduce
\begin{equation*}
\begin{split}
\|\widetilde{T}_ng\|_{L^k_xL^\infty_T}&=\|U_\beta(t)\int_0^TU_\beta(-t')P_n\partial_xg(\cdot,t')dt'\|_{L^k_xL^\infty_T}\\
& \leq
C2^{n((1/2+)-\beta/2)}\|D^{\beta/2}\int_0^TU_\beta(-t')P_n\partial_xg(\cdot,t')dt'\|_{L^2_x}\\
& \leq C2^{n((1/2+)-\beta/2)}\|D^{\beta/2}\int_0^TU_\beta(T-t')P_n\partial_xg(\cdot,t')dt'\|_{L^2_x}\\
& \leq C2^{n((1/2+)-\beta/2)}\|\partial_xg\|_{L^1_xL^2_T}.
\end{split}
\end{equation*}
The conclusion then follows from the Christ and Kiselev lemma (see \cite[Lemma 3]{MR}), as in \cite{KMR}.
\end{proof}

Now we are able to prove the following well-posedness result.

\begin{theorem}\label{localdbo}
Let $1<\beta<2$ and $k>2\beta$. For any $u_0 \in H^{\beta/2}(\mathbb{R})$,
there exist $T=T(\|u_0\|_{H^{\beta/2}})>0$ and a unique solution of the IVP
\eqref{dBO}, defined in the interval $[0,T]$, such that
\begin{equation}\label{b1.2}
u \in C([0,T];H^{\beta/2}(\mathbb{R})),
\end{equation}
\begin{equation}\label{b2.2}
\|\partial_xu\|_{L^\infty_xL^2_{T}} <\infty,
\end{equation}
and
\begin{equation}\label{b3.2}
\|u\|_{L^{k}_x L^\infty_{T}}<\infty.
\end{equation}
Moreover, for any $T'\in(0,T)$ there exists a neighborhood $U$ of $u_0$ in
$H^{\beta/2}(\mathbb{R})$ such that the map $\widetilde{u}_0\mapsto
\widetilde{u}(t)$ from $U$ into the class defined by
\eqref{b1.2}--\eqref{b3.2} is continuous.
\end{theorem}
\begin{proof}
As usual, we consider the integral operator
\begin{equation}\label{Psi}
\Psi(u)(t)=\Psi_{u_0}(u)(t):=U_\beta(t)u_0-\int_0^t
U_\beta(t-t')\partial_x(u^{k+1})(t')dt',
\end{equation}
and define the metric spaces
$$
\mathcal{Y}_T=\{ u \in C([0,T];H^{\beta/2}(\mathbb{R})); \,\,\,\, \trinorm u
\trinorm <\infty \}
$$
and
$$
\mathcal{Y}_T^a=\{ u \in \mathcal{Y}_T; \,\,\,\, \trinorm u \trinorm \leq a
\},
$$
with
\begin{equation*}
\begin{split}
\trinorm u \trinorm:= \|u\|_{L^\infty_TH^{\beta/2}_x}+
\|u\|_{L^{k}_xL^\infty_{T}} + T^{-\gamma}\|u_x\|_{L^{\infty}_xL^2_{T}}
\end{split}
\end{equation*}
where $\gamma >0$ is an arbitrarily small number and $a,T>0$ will be chosen
later. We assume $T\leq 1$.

First we estimate the $H^{\beta/2}$-norm of $\Psi(u)$. Let $u\in
\mathcal{Y}_T$. By using Lemma \ref{lemmalinear} (i) and (vi), and then
H\"older's inequality, we have
\begin{equation}\label{b1}
\begin{split}
 \|\Psi(u)(t) \|_{L^\infty_TL^2_{x}} & \leq   \|U_\beta(t)u_0\|_{L^\infty_TL^2_x}+
 \|\int_0^t
U_\beta(t-t')\partial_x(u^{k+1})(t')dt'\|_{L^\infty_TL^2_x}\\
& \leq C\|u_0\|_{L^2_x}+C\|u^ku_x\|_{L^1_xL^2_T}\\
&\leq
C\|u_0\|_{L^2_x}+CT^\gamma\|u\|_{L^k_xL^\infty_T}^kT^{-\gamma}\|u_x\|_{L^\infty_xL^2_T}\\
& \leq C\|u_0\|_{L^2_x}+CT^\gamma\trinorm u \trinorm^{k+1}.
\end{split}
\end{equation}
Also, using Lemma \ref{lemmalinear} (i) and (iv), and Holder's inequality,
we get
\begin{equation}\label{b2}
\begin{split}
 \|D^{\beta/2}&\Psi(u)(t) \|_{L^\infty_TL^2_{x}}\\
  & \leq   \|D^{\beta/2}U_\beta(t)u_0\|_{L^\infty_TL^2_x}+
 \|D^{\beta/2}\int_0^t
U_\beta(t-t')\partial_x(u^{k+1})(t')dt'\|_{L^\infty_TL^2_x}\\
& \leq C\|D^{\beta/2}u_0\|_{L^2_x}+C\|u^ku_x\|_{L^1_xL^2_T}\\
& \leq C\|u_0\|_{H^{\beta/2}}+CT^\gamma\trinorm u \trinorm^{k+1}.
\end{split}
\end{equation}

Now, in view of Lemma \ref{lemmalinear} (ii) and (v) (with $s=1$),
\begin{equation}\label{b3}
\begin{split}
 \|\partial_x\Psi(u)(t) \|_{L^\infty_xL^2_{T}} & \leq   \|\partial_xU_\beta(t)u_0\|_{L^\infty_xL^2_T}+
 \|\partial_x\int_0^t
U_\beta(t-t')\partial_x(u^{k+1})(t')dt'\|_{L^\infty_xL^2_T}\\
& \leq C\|u_0\|_{H^{\beta/2}}+C\|u^ku_x\|_{L^1_xL^2_T}\\
& \leq C\|u_0\|_{H^{\beta/2}}+CT^\gamma\trinorm u \trinorm^{k+1}.
\end{split}
\end{equation}

Finally, taking a number $1/2+<\beta/2$, an application of Corollary \ref{maximalcor} yields
\begin{equation}\label{b4}
\begin{split}
 \|\Psi(u)(t) \|_{L^k_xL^\infty_{T}} & \leq   \|U_\beta(t)u_0\|_{L^k_xL^\infty_T}+
 \|\int_0^t U_\beta(t-t')\partial_x(u^{k+1})(t')dt'\|_{L^k_xL^\infty_T}\\
& \leq C\|u_0\|_{H^{1/2+}}+C\|u^ku_x\|_{L^1_xL^2_T}\\
& \leq C\|u_0\|_{H^{\beta/2}}+CT^\gamma\trinorm u \trinorm^{k+1}.
\end{split}
\end{equation}

 Choose $a=6C
\|u_0\|_{H^{\beta/2}}$, and $T>0$ such that
$$
C\,a^k T^{\gamma} \leq \frac{1}{6}.
$$
Then, we see that $\Psi:\mathcal{Y}_T^a \mapsto \mathcal{Y}_T^a$ is well
defined. Moreover, similar arguments show that $\Psi$ is a contraction. The rest of the proof uses standard arguments, thus, we omit the details.
\end{proof}

\section{Global well-posedness in the energy space}

Let us start by recalling the following sharp Gagliardo-Nirenberg type inequality.

\begin{theorem}\label{bestd}
Let $k>0$ and $1<\beta<2$,  then the Gagliardo-Nirenberg inequality
\begin{equation}\label{g-n-d}
\|f\|_{L^{k+2}}^{k+2}\le K_{\rm opt}^{k+2}\,\|D^{\beta/2}
f\|_{L^2}^{k/\beta}\|f\|_{L^2}^{(2+(k+2)(\beta-1))/\beta},
\end{equation}
holds, and the sharp constant $K_{\rm opt}>0$ is
\begin{equation}\label{optd}
K_{\rm
opt}^{k+2}=\frac{(k+2)\beta}{2+(k+2)(\beta-1)}\left[\left(\frac{2+(k+2)(\beta-1)}{k}\right)^{1/\beta}
\frac{1}{\|Q\|_{L^2}^2}\right]^{k/2},
\end{equation}
where $Q$ is the unique non-negative, radially-symmetric, decreasing
solution of the equation
\begin{equation}\label{groundd}
D^\beta Q+Q-Q^{k+1}=0.
\end{equation}
\end{theorem}
\begin{proof}
The sharpness of the constant in \eqref{optd} was proved in \cite[Theorem
2.1]{ABLS}. The uniqueness of the non-negative, radially-symmetric,
decreasing (for $x>0$) solution of \eqref{groundd} was recently proved in \cite{FL} for the case $\beta \in (1,2)$. For $\beta=1$ it was proved by Amick and Toland \cite{AT}. 
\end{proof}

Let us establish some useful identities involving the ground state  $Q$.
First, by multiplying \eqref{groundd} by $Q$, integrating over $\mathbb{R}$,
we obtain
\begin{equation}\label{c1}
\int_{\R}Q^{k+2}\, dx=\|Q\|_{L^{2}}^{2}+\|D^{\beta/2} Q\|_{L^{2}}^{2}.
\end{equation}

Second, by multiplying \eqref{groundd} by $x\partial_x Q$, integrating, and
applying integration by parts, we obtain
\begin{equation}\label{c2}
\frac{2}{k+2}\int_{\R}Q^{k+2}\,
dx=\|Q\|_{L^{2}}^{2}-(\beta-1)\|D^{\beta/2}Q\|_{L^2}^2,
\end{equation}
where we have used the identity
$$
\int_{\mathbb{R}}x\partial_xQ(D^\beta
Q)dx=\frac{\beta-1}{2}\int_{\mathbb{R}}|D^{\beta/2}Q|^2dx.
$$
Combining  \eqref{c1} with \eqref{c2}, we conclude that
\begin{equation}\label{c3}
\frac{ k}{ 2+(k+2)(\beta-1)} \|Q\|_{L^{2}}^{2}=\|D^{\beta/2} Q\|_{L^{2}}^{2}
\end{equation}
and
\begin{equation}\label{c4}
\frac{2+(k+2)(\beta-1)}{k+2}
\int_\mathbb{R}Q^{k+2}\,dx=\beta\|Q\|_{L^{2}}^{2}.
\end{equation}

With these tools in hand, we are able to prove the following a priori
estimate.

\begin{proof}[Proof of Theorem \ref{global_dgbo}] We write the $\dot{H}^{\beta/2}$-norm of $u(t)$
using the quantities $M(u(t))$ and $E(u(t))$ and then use the sharp
Gagliardo-Nirenberg inequality \eqref{g-n-d} to get
\begin{equation}\label{ap10d}
\begin{split}
\|D^{\beta/2} u(t)\|_{L^2}^2&=2E(u_0)+\frac{2}{k+2} \int_{\R}u^{k+2}(t)\,dx\\
&\le 2E(u_0)+\frac{2}{k+2}K_{\rm opt}^{k+2}\,
\|u_0\|_{L^2}^{(2+(k+2)(\beta-1))/\beta}\|D^{\beta/2} u(t)\|_{L^2}^{k/\beta}
\end{split}
\end{equation}
Let $X(t)=\|D^{\beta/2} u(t)\|_{L^2}^2$, $A=2E(u_0)$, and
$$
B=\frac{2}{k+2}K_{\rm opt}^{k+2}\, \|u_0\|_{L^2}^{(2+(k+2)(\beta-1))/\beta}.
$$
Therefore,  we can write \eqref{ap10d} as
\begin{equation}\label{ap12d}
X(t)-B\,X(t)^{k/2\beta}\le A, \text{\hskip2pt for}\;\;t\in (0,T).
\end{equation}

Now let $f(x)=x-B\,x^{k/2\beta}$, for $x\ge 0$. The function $f$ has a local
maximum at $x_0=\Big(\dfrac{2\beta}{kB}\Big)^{2\beta/(k-2\beta)}$ with
maximum value
$f(x_0)=\dfrac{k-2\beta}{k}\Big(\dfrac{2\beta}{kB}\Big)^{2\beta/(k-2\beta)}.$
If we require that
\begin{equation}\label{ap13d}
2E(u_0) < f(x_0)\,\,\,\, \mbox{and}\, \,\,\, X(0) < x_0,
\end{equation}
the continuity of $X(t)$ implies that $X(t) < x_0$ for any $t$ as long as
the solution exists.

Using \eqref{c3} and \eqref{c4}, we immediately  deduce that
\begin{equation}\label{EQ}
E(Q)=\frac{1}{2}\dfrac{k-2\beta}{2+(k+2)(\beta-1)}\|Q\|_{L^2}^2.
\end{equation}
Therefore, a simple calculation shows that conditions (\ref{ap13d}) are
exactly the inequalities (\ref{GR1d}) and (\ref{GR2d}). Indeed,
\begin{equation*}
\begin{split}
2E(u_0)<f(x_0)\quad \Leftrightarrow\quad
2E(u_0)<\dfrac{k-2\beta}{k}\Big(\dfrac{2\beta}{kB}\Big)^{2\beta/(k-2\beta)}\\
\end{split}
\end{equation*}
Using the explicit form of $K_{\rm opt}^{k+2}$ in \eqref{optd}, we have
$$
\Big(\dfrac{2\beta}{kB}\Big)^{2\beta/(k-2\beta)}=\frac{k}{2+(k+2)(\beta-1)}
\frac{\|Q\|_{L^2}^{2k\beta/(k-2\beta)}}{\|u_0\|_{L^2}^{2(2+(k+2)(\beta-1))/(k-2\beta)}}.
$$
Thus,
\begin{equation*}
\begin{split}
&2E(u_0)<f(x_0)\\
&\Leftrightarrow \\
&E(u_0)\|u_0\|_{L^2}^{2(2+(k+2)(\beta-1))/(k-2\beta)}<\dfrac{(k-2\beta)\|Q\|_{L^2}^{2+2(2+(k+2)(\beta-1))/(k-2\beta)}}{2(2+(k+2)(\beta-1))}\\
&\Leftrightarrow E(u_0)\|u_0\|_{L^2}^{2(2+(k+2)(\beta-1))/(k-2\beta)}
<E(Q)\|Q\|_{L^2}^{2(2+(k+2)(\beta-1))/(k-2\beta)}\\
&\Leftrightarrow E(u_0)^{s_k} M(u_0)^{\beta/2-s_k} < E(Q)^{s_k}
M(Q)^{\beta/2-s_k}.
\end{split}
\end{equation*}

One see that the second inequality in (\ref{ap13d}) is equivalent to (\ref{GR2d}) in a similar way.
Moreover the inequality $X(t) < x_0$ reduces to (\ref{GR3d}). The proof of
Theorem \ref{global_dgbo} is thus completed.
\end{proof}

\vspace{3mm} \noindent{\large {\bf Acknowledgments}}
\vspace{3mm}\\  L.G.F. was partially supported by CNPq and FAPEMIG/Brazil. F. L. was partially supported
by CNPq and FAPERJ/Brazil.  A. P. was partially supported by CNPq/Brazil.

\end{document}